\documentclass[11pt]{amsart}
\usepackage[cm]{fullpage}
\usepackage{amssymb,amsmath,amsthm,amscd,mathrsfs,graphicx}
\usepackage[cmtip,all]{xy}
\usepackage{pb-diagram}
\usepackage{tikz}

\numberwithin{equation}{section}
\newtheorem{teo}{Theorem}[section]

\newtheorem{theorem}{Theorem}[section]

\newcommand{\ov}[1]{\overline{#1}}

\newcommand{\ve}{\varepsilon}

\theoremstyle{definition}

\theoremstyle{remark}
\newtheorem{remark}[teo]{Remark}

\begin{document}
	\bibliographystyle{amsplain}

	\title{Smoothness up to the free boundary for  the  $\boldsymbol{p}$-Laplacian  evolution equation \\ and the $\boldsymbol{\alpha}$-Gauss curvature flow}
	
	\author[A. Chau]{Albert Chau}
\address{Department of Mathematics, The University of British Columbia, 1984 Mathematics Road, Vancouver, B.C.,  Canada V6T 1Z2 \\ email: chau@math.ubc.ca} 
\author[B. Weinkove]{Ben Weinkove}
\address{Department of Mathematics, Northwestern University, 2033 Sheridan Road, Evanston, IL 60208, USA. \\ email: weinkove@math.northwestern.edu}

\thanks{Research supported in part by  NSERC grant $\#$327637-06, NSF grant DMS-2348846 and the Simons Foundation}

	\vspace{-20pt}

\begin{abstract} The $p$-Laplacian evolution equation and the $\alpha$-Gauss curvature flow with a flat side are degenerate parabolic equations with evolving free boundaries.  We give proofs of  smooth short-time existence, up to the free boundaries, using a result of the authors on  linear degenerate equations on a fixed domain.  
\end{abstract}
	\maketitle
	
	\section{Introduction}

In a previous work \cite{CW}, the authors showed existence of smooth solutions to linear degenerate parabolic equations on a fixed domain. The key feature is that no boundary conditions are imposed on the solution; instead one imposes certain degeneracy conditions for the operator on the boundary.  The proof used an approach of  Kohn-Nirenberg \cite{KN}. This was used to give a proof of a  short-time existence result for the porous medium equation  $u_t = \Delta(u^m)$ for $1<m\le 2$ (cf. \cite{DH98, Koch}), giving smoothness up to the boundary of the set $\{ u>0\}$.

In this paper, we build on \cite{CW} by showing that the same method can be applied to two other nonlinear degenerate parabolic equations:  the $p$-Laplacian  evolution equation  and the $\alpha$-Gauss curvature flow with a flat side.  This shows the versatility of this method, which we believe should apply to other settings too. With this in mind, we carry out our calculations in a way that they can easily be adapted to other equations.

While the smooth short-time existence results we give below  are not entirely new, we provide here a different and unified perspective.  The method is straightforward and direct:  by a standard change of variables, one can reduce the problem on a moving domain to an equation on a fixed domain with no boundary conditions.  The linearization of this equation is a linear degenerate parabolic equation.  One only needs to check that the coefficients of this linearized equation satisfy certain conditions on the boundary including the key condition of Fichera (see (A) and (B) in Section \ref{SectionCW} below).  Then one can apply the result of \cite{CW} which gives short-time smooth existence using  the Nash-Moser inverse function theorem.  Since the existence of unique weak solutions in these settings is standard, we state our results in the form of regularity theorems.  

Previous approaches in the literature instead use specific model equations and Schauder estimates \cite{DH98, DH,  HWZ2, JL, Koch, KLR, LPV}.

We now describe our two main results.  The $p$-Laplacian evolution equation is the degenerate parabolic equation for $u \ge 0$,
	\begin{equation} \label{pLap}
		u_t = \Delta_p u := \nabla \cdot (|\nabla u|^{p-2} \nabla u), \quad p>2.
	\end{equation}
It can be regarded as a way of ``non-linearizing'' the heat flow, which corresponds to the case $p=2$.   It is used to model diffusion processes in non-Newtonian fluids, turbulent flows in porous media, and image processing.  
	
	Note that
	\begin{equation} \label{Deltap}
		\begin{split}
			\Delta_p u  
			= {} & | \nabla u|^{p-2} \Delta u + (p-2) |\nabla u|^{p-4} u_i u_j u_{ij}.
		\end{split}
	\end{equation}
It is convenient to work with the \emph{pressure}
	$$g = \left( \frac{p-1}{p-2} \right) u^{\frac{p-2}{p-1}},$$
	which satisfies
	\begin{equation} \label{plapg}
		g_t = \left( \frac{p-2}{p-1} \right) g \Delta_p g + |\nabla g|^p.
	\end{equation}

	We consider the case when the initial data $g_0$ is  a smooth positive function on a bounded domain $\Omega_0 \subset \mathbb{R}^n$ and vanishes on the boundary $\partial \Omega_0$, which we assume to be smooth.  In addition we make the non-degeneracy assumption
	\begin{equation} \label{nondegen}
		g_0 + (\nabla g_0)^2 \ge c>0, \quad \textrm{on } \ov{\Omega}_0,
	\end{equation}
	for a uniform constant $c>0$ (equivalently, $|\nabla g_0| \neq 0$ on $\partial \Omega_0$).  There exists a unique continuous weak solution $g=g(x,t)$ to (\ref{plapg}) on $\mathbb{R}^n \times [0,\infty)$  (see \cite{DBH} and also \cite{DB} for the definition of a weak solution and proof of its existence).   
	  We write $$\Omega_t = \{ x \in \mathbb{R}^n \ | \ g(x,t)>0\}.$$   Our main result in this case proves that the free boundaries $\partial \Omega_t$ are smooth and the solution is smooth in a neighborhood of the boundary.  
	 	
	\begin{theorem} \label{thmplap}
	There exists $T>0$ such that for $t\in [0,T]$, the sets $\Omega_t$ have smooth boundaries $\partial \Omega_t$.  The function $g=g(x,t)$ on $\displaystyle{\cup_{t \in [0,T]} \ov{\Omega}_t \times \{ t\} }$ is smooth in a neighborhood of the boundary $\displaystyle{\cup_{t \in [0,T]} \partial \Omega_t \times \{ t\} }$, satisfying (\ref{plapg}) there in the classical sense.
\end{theorem}

Existing regularity results were established using different methods.
Under weaker hypotheses on the initial data, Lee-Petrosyan-V\'azquez \cite{LPV}  showed that $g$ is smooth in a neighborhood of the boundary for $t>0$, using estimates of Koch \cite{Koch}.  Earlier work had established Lipschitz regularity for solutions \cite{CK, ZY} and $C^{1, \gamma}$ regularity of the free boundary \cite{Ko}.

	Next we consider the $\alpha$-Gauss curvature flow of convex hypersurfaces in $\mathbb{R}^{n+1}$.  This is a subject with a long history \cite{A, ACGL, BCD, Chow, Tso}.  We consider the case where the initial hypersurface has a flat side that persists.  
	
	Let $\Sigma$ be a closed convex hypersurface in $\mathbb{R}^{n+1}$ parametrized by a map $X_0: S^n \rightarrow \mathbb{R}^{n+1}$.  Then the $\alpha$-Gauss curvature flow starting at $\Sigma$ is a family of maps $X(t) : S^n \rightarrow \mathbb{R}^{n+1}$ satisfying
	\begin{equation} \label{aGCF}
	\frac{\partial X}{\partial t} = - K^{\alpha} \nu, \quad X(0)=X_0,
	\end{equation}
	where $K$ is the Gauss curvature and $\nu$ is the outward pointing unit normal.
	We make the assumption that $\Sigma$ has a flat side whose boundary is strongly convex, meaning that the principal curvatures are strictly positive.  Write 
	$\Sigma = \Sigma_1\cup \Sigma_2$ where the open strongly convex set $\Sigma_1$ is the interior of the flat side which we assume lies in the plane $x_{n+1}=0$ and contains the origin.  We assume that $\Sigma$ lies in the half space $x_{n+1} \ge 0$.

The case of the $\alpha$-Gauss curvature flow with a flat side was first studied by Hamilton \cite{Hamilton}.  
We  make the assumption 
	$$\alpha>1/n,$$
	which is necessary for the flat side to persist for some positive time \cite{Andrews, Hamilton}.  There is a unique  $C^{1, \gamma}$ family of convex hypersurfaces $\Sigma^t$ solving the $\alpha$-Gauss curvature flow (\ref{aGCF}) \cite{DS} which are smooth in the strictly convex part \cite{Chow, Tso}.
	
	Fix a neighborhood $W_0$ of $\ov{\Sigma}_1$ in $\mathbb{R}^n$ (identifying $\mathbb{R}^n$ with $x_{n+1}=0$), so that in $\ov{W}_0$, the part of the hypersurface $\Sigma^t$ near $x_{n+1}=0$ can be written as a graph $x_{n+1} = u(x, t),$ for $x=(x_1, \ldots, x_n)$ over $W_0$.  It is convenient to make a change of variables and consider
	$$g =\left(  \frac{\sigma +1}{\sigma} u \right)^{\frac{\sigma}{\sigma+1}}, \quad \sigma:= n-\frac{1}{\alpha}>0, \quad \frac{2}{\sigma} \in \mathbb{Z}^+,$$
	where $\mathbb{Z}^+$ denotes the set of positive integers. 
 As pointed out in \cite{HWZ2}, the condition $2/\sigma\in \mathbb{Z}^+$ is necessary in general for smoothness up to the boundary, and in particular under the following non-degeneracy assumption at the boundary.
	We make the assumption that the initial data $g_0(x):=g(x,0)$ is smooth on $\ov{W}_0 \setminus \Sigma_1$ and satisfies
	$$|\nabla g_0| \ge c> 0, \quad \textrm{on } \partial \Sigma_1.$$

With these assumptions, we prove a smooth short-time existence result, which gives smoothness of the free boundaries $\partial \Sigma_1^t$ and of the graphs of $g$ in a neighborhood of the free boundary.  This gives a different proof of a result of Huang-Wang-Zhou \cite{HWZ2}.

\begin{theorem} \label{thmGCF}
There exists $T>0$ such that for $t \in [0,T]$, the solution $\Sigma^t$ of (\ref{aGCF}) can be written as $\Sigma^t = \Sigma_1^t \cup \Sigma_2^t$, where the flat side $\Sigma_1^t$ is a domain in $x_{n+1}=0$.  The domains $\Sigma_1^t \subset \mathbb{R}^n$ have smooth boundaries $\partial \Sigma_1^t$ which vary smoothly with $t$.  The function $g=g(x,t)$ on $\displaystyle{\cup_{t\in[0,T]} (\ov{W}_0 \setminus \Sigma_1^t )\times \{t \}}$ is smooth. 
\end{theorem}

In the case $\alpha=1$ (Gauss curvature flow) and $n=2$, smoothness at the boundary was first obtained for $t>0$, under weaker hypotheses of the initial data, by Daskalopoulos-Hamilton \cite{DH} using techniques from \cite{DH98}. This was extended to to general $n$ and $\alpha$ satisfying the above conditions by Huang-Wang-Zhou \cite{HWZ2} (see also \cite{HWZ, KLR}).
%to $n=2$ and $p \in (1/2, 1]$ by Kim-Lee-Rhee \cite{KLR} and to $n>2$  under some additional conditions by Huang-Wang-Zhou \cite{HWZ2} (see also \cite{HWZ}).

The outline of the paper is as follows.  In Section \ref{SectionCW} we briefly describe the authors' previous result  \cite{CW} on smooth short-time existence for degenerate equations on a fixed domain.  In Section \ref{Sectionlinearization}, we describe the change of variables arguments for the cases of the $p$-Laplacian evolution equation and the $\alpha$-Gauss curvature flow, and then carry out a rather general computation for the linearization of these equations.  In Sections \ref{sectionp} and \ref{SectionGCF} we complete the proofs of Theorems \ref{thmplap} and \ref{thmGCF}, respectively.

	\section{Short time existence and boundary conditions for nonlinear equations} \label{SectionCW}
	
	In this section we describe the result we need from \cite{CW} on the existence of short time solutions to nonlinear equations on a fixed domains under certain conditions on the coefficients at the boundary.  
	
	Let $\Omega_0$ be a bounded domain in $\mathbb{R}^n$ with smooth boundary $\partial \Omega_0$.  For a small $T>0$ we consider equations of the form
		\begin{equation} \label{hG}
	\begin{split}
	\frac{\partial h}{\partial t} = {} &  G(D^2 h, Dh, h, x), \quad \textrm{in } \ov{\Omega}_0 \times [0,T] \\
	h(x,0) = {} & h_0(x), \quad \textrm{for } x \in \ov{\Omega}_0,
	\end{split}
	\end{equation}
	where $G$ is a smooth nonlinear operator, and initial data $h_0 \in C^{\infty}(\ov{\Omega}_0)$.  The result of \cite{CW} treats the special case when $h_0=0$ but the same argument applies here as we now describe.
	
	 The space $C^{\infty}(\ov{\Omega}_0 \times [0,T])$
 is a Frechet space equipped with seminorms $$\| f \|_{\ell} = \max_{2k+|\mathbf{a}| \le \ell} \, \, \max_{\ov{\Omega}_0 \times [0,T]} |D_x^{\mathbf{a}} D_t^k f|.$$  These define a metric on $C^{\infty}(\ov{\Omega}_0 \times [0,T])$ by the formula
 \begin{equation} \label{d}
 d(f,g) = \sum_{k=0}^{\infty} 2^{-k} \frac{\| f-g \|_k}{1+\| f-g\|_k}.
 \end{equation}
 Consider the linear space 
 $$\mathcal{B} = \{ f \in C^{\infty}(\ov{\Omega}_0 \times [0,T]) \ | \ (D^k_t f)(x,0)=0, \ \textrm{for all } k\ge 0 \},$$
and  let $\mathcal{D}$ be a ball of small radius $\sigma>0$ centered at $0 \in \mathcal{B}$, with respect to the metric (\ref{d}).
Let $\overline{h} \in C^{\infty}(\ov{\Omega}_0 \times [0,T])$ be a ``formal solution'' of 
$\Phi(h):= h_t - G(D^2 h, Dh,h, x)=0$  at $t=0$, meaning that it solves 
$D_t^j \Phi(\ov{h})(x,0)=0$ for all $x\in \Omega_0$,  
$j=0,1,2,\ldots$ and $\ov{h}(x,0)=h_0(x)$.  Its Taylor series expansion in $t$ is determined by the equation  (as in \cite[Section 4]{CW}) and the existence of $\ov{h}$ follows from the existence of smooth functions with a given Taylor series.

	For $h = \ov{h} + \zeta$ with $\zeta \in \mathcal{D}$, let $DG_h$ be the linearization of $G$ at $h$, which we write as
	$$DG_h (w) = a^{ij} w_{ij} + b^i w_i + f w,$$
	for coefficients $a^{ij}$, $b^i$ and $f$ which are smooth functions of $h$, $Dh$, $D^2h$ and $x$.  We may and do assume that $(a^{ij})$ is symmetric.  We assume that for all $h = \ov{h} + \zeta$, with $\zeta \in \mathcal{D}$, the following conditions hold.
	
	First, we have
ellipticity in the interior of the domain, namely:
	\begin{equation} \tag{E}
	(a^{ij}) >0, \quad \textrm{on } \Omega_0
	\end{equation}
	 We assume degeneracy of $(a^{ij})$ on the boundary, in the sense that:
	 	 \begin{equation} \tag{A} 
a^{ij}\nu_i\nu_j  =0, \quad  \textrm{ on }\partial \Omega_0,
\end{equation}
for $\nu=\nu(x)$ the outward pointing normal at $x \in \partial \Omega_0$.  
We also impose the (strict) Fichera \cite{Fi56} condition:
\begin{equation} \tag{B} 
(b^i - \partial_j a^{ij})\nu_i < 0,  \quad  \textrm{ on }\partial \Omega_0.
\end{equation}
Then the results of \cite{CW} imply the existence of solutions to (\ref{hG}) for a short time:

\begin{theorem} \label{CWthm}
Assuming the above, there exists $\ve>0$ and $h(x,t) \in C^{\infty}(\ov{\Omega}_0 \times [0,\ve])$ solving
\[\begin{split}
	\frac{\partial h}{\partial t} = {} &  G(D^2 h, Dh, h, x), \quad \textrm{in } \ov{\Omega}_0 \times [0,\ve] \\
	h(x,0) = {} & h_0(x), \quad \textrm{for } x \in \ov{\Omega}_0.
	\end{split}\]
\end{theorem}
	
	This is proved using the existence of smooth solutions for the linearized equation, together with the Nash-Moser inverse function theorem.
	
	\begin{remark} This result also holds with the strict inequality $<$ in (B) replaced by $\le$, at the expense of requiring the condition 
	\begin{equation} \tag{A$_2$} 
	 (\partial_k a^{ij})\nu_k \nu_{i} \nu_j <0, \quad  \textrm{ on }\partial \Omega_0,
\end{equation}
see \cite[Remark 1.1(i)]{CW}.  Either approach works for the main results of this paper, since it is easily seen that (A$_2$) holds for our equations.
\end{remark}
	
	\section{Change of variables and linearization} \label{Sectionlinearization}

	In the two cases, the $p$-Laplacian evolution equation and the $\alpha$-Gauss curvature flow with a flat side, we will describe how we transform the equation to a fixed domain.  In each case we will assume we have a smooth solution of the original equation and show how it gives rise to a smooth solution of a different equation on the fixed domain.  This process is reversible.  Similar transformations are carried out in \cite{DH, Hanz, Koch}, for example. 
	
	\subsection{The $p$-Laplacian evolution equation} \label{subsectionplap}

	As in the introduction, $\Omega_0$ is a fixed domain in $\mathbb{R}^n$ with smooth boundary.  Fix $\delta>0$ and write $\Omega_0^{(\delta)}$ for the open set of points in $x \in \Omega_0$ such that $d(x, \partial \Omega_0)>\delta$.  Suppose that for $t \in [0,T]$ we have bounded open sets $\Omega_t \subset \mathbb{R}^n$ with smooth, and smoothly varying, boundaries $\partial \Omega_t$.
	Let $g=g(x,t)$ be nonnegative smooth functions satisfying the nonlinear equation, for $F$ defined by the right hand side of (\ref{plapg}),
	\begin{equation} \label{gF}
	g_t = F(\nabla^2 g, \nabla g, g),
	\end{equation}
	on $\bigcup_{t \in [0,T]} U_t \times \{ t\}$, for moving sets $U_t = \ov{\Omega}_t \setminus \Omega_0^{(\delta)}$.

It follows immediately from the formula (\ref{plapg}) that
 $F=F( (A_{ij}), p_i, u)$ 
	 is smooth with $\frac{\partial F}{\partial A_{ij}} \ge 0$ for $x \in \ov{\Omega}_t$ and  $\frac{\partial F}{\partial A_{ij}} > 0$ on the open set $\Omega_t$. We also have the following degeneracy of the operator at points of $\partial \Omega_t$:
	 \begin{equation} \label{deg}
\frac{\partial F}{\partial A_{ij}} = 0, \quad i,j=1, \ldots, n.
	  \end{equation}

We will now transform (\ref{gF}) to an equation for a function $h$ on a fixed neighborhood of $\partial \Omega_0$ in $\ov{\Omega}_0$.

Write $v:= g_0$, so that 
	$v: \ov{\Omega}_0 \rightarrow [0,\infty)$
	is a fixed smooth function on the set $\ov{\Omega}_0$, which vanishes on $\partial \Omega_0$ and is positive on $\Omega_0$ and such that the 
	nondegeneracy condition 
	\begin{equation} \label{nondegen0}
		v + (\nabla v)^2 \ge c>0, \quad \textrm{on } \ov{\Omega}_0,
	\end{equation}
	holds for a uniform constant $c>0$.

	Define a vector-valued map $V: \Omega_0  \rightarrow \mathbb{R}^{n}$ by
	\begin{equation} \label{V}
	V(x):= (-\nabla v (x), v(x)),
	\end{equation}
	which extends to a smooth map on $\ov{\Omega}_0$.	
	
	For $x \in \Omega_0$, the vector $V(x)$ is transverse to the graph $x_{n+1}=v(x)$ in $\mathbb{R}^{n+1}$ at the point $(x, v_0(x))$.  Note that $V(x)$ is parallel to the $x_{n+1}=0$ hyperplane for $x \in \partial \Omega_0$.  Moreover, by (\ref{nondegen0}), we have $|V| \ge c'>0$.  
	
For points $x \in \ov{\Omega}_0$ which are sufficiently close to the boundary $\partial \Omega_0$,
	the parametrized line segment $$h \mapsto (x, v(x)) + hV(x), \quad h\in \mathbb{R},$$ intersects the graph $x_{n+1}=g(x,t)$ for $t \in [0,T]$ at a unique value of $h$, where we are shrinking $T>0$ if necessary.  This defines a function $h=h(x,t)$.  
	The function $h(x,t)$ satisfies
	\begin{equation} \label{h0}
		(1+  h(x,t) )v(x) = g(x-h(x,t) \nabla v(x), t),
	\end{equation}
	and indeed (\ref{h0})  can be used to define $h$.  We define a  time-dependent diffeomorphism,
	\begin{equation} \label{defnPhi}
	\Phi (x,t) = x- h(x,t) \nabla v(x).
	\end{equation}
	on $\Omega_0 \setminus \Omega_0^{(\eta)}$ for some $\eta>0$.
	Write $\Psi$ for the inverse of the diffeomorphism defined by (\ref{defnPhi}), so that $\Psi(\Phi(x,t),t) = x$.
	
	Now $h(x,t)$, for $t \in [0,T]$ satisfies a second order nonlinear equation 
		\begin{equation} \label{dhdt1}
	\frac{\partial h}{\partial t} = G(D^2h, Dh, h,x),
	\end{equation}
	for $x \in \ov{\Omega}_0 \setminus \Omega_0^{(\eta)},$ for a small $\eta>0$, and satisfies $h(x,0)=0$.
	
	In the above, we assumed \emph{a priori} that $g$ was smooth, but under this transformation, if $g$ is only continuous, it maps to a continuous function $h$.
	
	\subsection{The $\alpha$-Gauss curvature flow with a flat side} We consider the $\alpha$-Gauss curvature flow of convex hypersurfaces in $\mathbb{R}^{n+1}$ with a flat side, as described in the introduction.  Locally, this flow is given by a flow of graphs $x_{n+1}=u(x,t)$ for $x=(x_1, \ldots, x_n)$ by the equation
	\begin{equation} \label{ueqn}
		u_t = \frac{(\det D^2 u)^{\alpha}}{(1+|Du|^2)^{\frac{(n+2)\alpha-1}{2}}}.
	\end{equation}
	As in the introduction, we write	$$g =\left(  \frac{\sigma +1}{\sigma} u \right)^{\frac{\sigma}{\sigma+1}}, \quad \sigma:= n-\frac{1}{\alpha}.$$
	Then compute
	$$u_t = g^{\frac{1}{\sigma}} g_t, \quad u_i = g^{\frac{1}{\sigma}} g_i,$$
	$$u_{ij} = g^{\frac{1}{\sigma}} g_{ij} + \frac{1}{\sigma} g^{\frac{1}{\sigma}-1} g_i g_j.$$
One can check then that $g$ solves
\begin{equation} \label{eqnGCF}
g_t = F(\nabla^2 g, \nabla g, g),
\end{equation}
for
\begin{equation} \label{eqnGCF2}
F(\nabla^2 g, \nabla g, g)=\frac{\left(  \det (g^{\frac{1}{n}} g_{ij} + \frac{1}{\sigma} g^{\frac{1}{n}-1} g_i g_j) \right)^{\alpha}}{(1+g^{\frac{2}{\sigma}} |\nabla g|^2)^{\frac{(n+2)\alpha-1}{2}}},
\end{equation}
noting that this is of course different from the operator $F$ in (\ref{gF}), even though we use the same letter. 

From this formula,
 $F=F( (A_{ij}), p_i, u)$ 
	 is smooth with $\frac{\partial F}{\partial A_{ij}} \ge 0$ for $x \in \ov{W}_0\setminus \Sigma^t_1$ and  $\frac{\partial F}{\partial A_{ij}} > 0$ on the open set $\ov{W}_0\setminus \ov{\Sigma}^t_1$.   The condition  $2/\sigma\in \mathbb{Z}^+$ is used here to ensure the smoothness of $F$ when derivatives land on the denominator.	 We have the following degeneracy of the operator at points $x$ of $\partial \Sigma_1^t$, for $\nu_t$ the outward pointing unit normal at $x$,
	 \begin{equation} \label{degn}
\nu_t^i \frac{\partial F}{\partial A_{ij}} = 0, \quad i=1, \ldots, n.
	  \end{equation}
	  This will be proved below in Section \ref{SectionGCF}.

We suppose that for small $T>0$, we have a family of domains $\Sigma_1^t \subset \mathbb{R}^n$ for $t \in [0,T]$ with smooth boundaries $\partial \Sigma_1^t$ which vary smoothly with $t$.  Using the terminology of the introduction, let $g(x,t)$ be smooth on $\cup_{t\in [0,T]} (\ov{W}_0 \setminus \Sigma_1^t) \times \{t\}$ solving (\ref{eqnGCF}).

The initial data $g_0$ is defined on $\ov{W}_0 \setminus \Sigma_1$.  In this section, we will write $\Omega_0:= \Sigma_1$, a convex domain in $\mathbb{R}^n$ with smooth boundary $\partial \Omega_0$.
	Let
	$v: \ov{\Omega}_0 \rightarrow [0,\infty)$
	be a fixed smooth concave function on the convex set $\ov{\Omega}_0$, which vanishes on $\partial \Omega_0$ and is positive on $\Omega_0$ and such that the 
	nondegeneracy condition 
	\begin{equation} \label{nondegen}
		v + |\nabla v|^2 \ge c>0, \quad \textrm{on } \ov{\Omega}_0,
	\end{equation}
	holds for a uniform constant $c>0$.  Equivalently, we have $|\nabla v| >0$ on $\partial \Omega_0$.

  	As in the case of the $p$-Laplacian evolution equation above, we define a vector-valued map $V$ by (\ref{V}).  For points $x \in \ov{\Omega}_0$ close enough to the boundary, the line segment $h \mapsto (x, v(x))+ hV(x)$ intersects the graph $x_{n+1} = g(x,t)$ at a unique value $h$, defining $h=h(x,t)$ which satisfies (\ref{h0}).  The difference here is that the graph $x_{n+1}=g(x,t)$ is defined for $x$ in the \emph{exterior} of $\Omega_0$.  In particular, we  no longer have $h(x,0)=0$ for all $x \in \ov{\Omega}_0$, although we do have $h(x,0)=0$ for $x\in \partial \Omega_0$.  We still have that $h(x,t)$ for $t \in [0,T]$ satisfies a second order nonlinear equation 
	\begin{equation} \label{dhdt2}
	\frac{\partial h}{\partial t} = G(D^2h, Dh, h,x),
	\end{equation}
	for $x \in \ov{\Omega}_0 \setminus \Omega_0^{(\eta)}$ for a small $\eta>0$.  Of course, the operator $G$ here differs from the operator described in Section \ref{subsectionplap}, even though we use the same letter.
	
	\subsection{Computing the linearization} In order to apply Theorem \ref{CWthm}, we need to compute the linearization of the equations (\ref{dhdt1}) and (\ref{dhdt2}) at a function $h = \ov{h} + \zeta$ for $\zeta \in \mathcal{D}$ (using the terminology of Section \ref{SectionCW} above).
	
	 The variation $\delta h$ of $h$ will satisfy a linear equation
	$$(\delta h)_t = a^{ij} (\delta h)_{ij} + b^i (\delta h)_{i} + f \delta h,$$
	for coefficients $a^{ij}, b^i, f$ depending on $x$ and $t$.
	
	The main result of this section is the following theorem relating these coefficients at a boundary point $x_0 \in \partial \Omega_0$ and $t=0$ to quantities involving the corresponding function $g$ from \eqref{h0}.  We do this in a general way which applies to both the $p$-Laplacian evolution equation \emph{and} the $\alpha$-Gauss curvature flow.  
	
	We choose local coordinates centered at $x_0$ so that the outward unit normal to $\Omega_0$ at $x_0$ is given by $\nu= (0, \ldots, 0, 1)$.  We also assume that locally $\partial \Omega_0$ is given by $x_n = \rho(x')$ for $x'=(x_1, \ldots, x_{n-1})$, where $\rho$ is a smooth function with $\rho(0)=0$ and $\nabla \rho(0)=0$.	
	
	\begin{theorem} \label{thmtechnical} At $x_0$ we have 
	\begin{enumerate}
	\item[(i)] $a^{nn} = \frac{1}{g_n} \frac{\partial F}{\partial A_{ij}} \Psi^n_j \Psi^n_i \Psi^n_n v_n (1+h)$ for $t \in [0,T]$.
	\end{enumerate}
	At  $x_0$ and $t=0$, we have:
	\begin{enumerate}
	\item[(ii)] $\displaystyle{\partial_n a^{nn} = \left( \partial_n \frac{\partial F}{\partial A_{nn}} \right) (\Psi^n_n)^2}$.
	\item[(iii)] $\displaystyle{b^n = \sum_{i,j=1}^{n-1}  \frac{\partial F}{\partial A_{ij}} \Psi^n_n v_n h_{ij} + \frac{\partial F}{\partial p_n}  \Psi^n_n - \frac{F}{v_n}}$.
	\item[(iv)] Finally, 
	\[ \begin{split} b^n - \sum_{i=1}^n \partial_i a^{ni} = {} &   \sum_{i,j=1}^{n-1}  \frac{\partial F}{\partial A_{ij}}  \Psi^n_n v_n h_{ij} + \frac{\partial F}{\partial p_n} \Psi^n_n - \frac{F}{v_n}  \\ {} & - 2\sum_{i=1}^{n-1} \left( \partial_i \frac{\partial F}{\partial A_{in}} \right) \Psi^n_n - \left( \partial_n \frac{\partial F}{\partial A_{nn}} \right) (\Psi^n_n)^2. \end{split} \]
	\end{enumerate}
	\end{theorem}
	\begin{proof}
Recall equation (\ref{h0}), 
	\begin{equation} \label{h1}
	g(\Phi, t) = v(1+h).
	\end{equation}
	From the definition of $\Phi$ and $\Psi$, we have
$\Psi^k_i \Phi^i_{\ell} = \delta_{k\ell} = \Phi^k_i \Psi^i_{\ell}$,
where
\begin{equation} \label{Phi0}
\Phi^i_{\ell} = \delta_{i \ell} - h_{\ell} v_i - h v_{i \ell}.
\end{equation}

	Differentiating (\ref{h1}),
	\begin{equation} \label{gi0}
	\begin{split} 
	g_t =  {} & h_t (v+ g_i v_i) \\
	g_i = {} & \Psi^k_i (v_k (1+h)+ v h_k) \\
	g_{ij} = {} & \Psi^k_{ij} (v_k (1+h) + v h_k) + \Psi^{\ell}_j \Psi^k_i (v_{k\ell} (1+h) + v_k h_{\ell} + v_{\ell} h_k + v h_{k\ell}) 
	\end{split}
	\end{equation}
	where we are always evaluating $g$, $\Psi$  and their derivatives at $\Phi(x,t)$. 
	
	It follows from the above that one can obtain the operator $G$ acting on $h$ by
	\begin{equation} \label{GF}
	G(D^2h, Dh, h, x)= \frac{F(D^2 g, Dg, g)}{v+ g_i v_i}.
	\end{equation}
	
	Regarding $\Phi$, $\Psi$, $g$ and its derivatives as functions of $h$ we wish to compute their variations $\delta \Phi$, $\delta \Psi$ and $\delta g$ at the point $x_0$.  Recall that at $x_0$ we have $v=g=0$  and we use coordinates so that $\nu = (0, \ldots, 0,1)$ so that $g_i=v_i=0$ for $i=1, \ldots, n-1$.  
	Since we are only interested in those terms which involve $(\delta h)_{in}$ for $i=1, \ldots, n$ and $(\delta h)_n$, we will neglect terms involving $\delta h$, $(\delta h)_a$ or $(\delta h)_{ab}$ for $1 \le a, b \le n-1$.  In addition, we neglect any terms involving only $(\delta h)_i$ whose coefficients vanish at $x_0$ or terms $(\delta h)_{ij}$ whose coefficients vanish to order at least two.  We use $\approx$ to indicate that we are neglecting these terms.

From  (\ref{Phi0}),
\begin{equation}
\delta \Phi^i_{\ell} \approx - (\delta h)_{\ell} v_i.
\end{equation}

Compute
\begin{equation} \label{Psiki}
\delta \Psi^k_i \approx \Psi^j_i \Psi^k_{\ell}  (\delta h)_j v_{\ell}
\end{equation}
and
\begin{equation} \label{and}
\Psi^k_{ij} = \Psi^q_j \Psi^{\ell}_i \Psi^k_p (h_{\ell q} v_p + h_{\ell} v_{pq} + h_q v_{p\ell} + h v_{pq\ell}).
\end{equation}
From 
(\ref{GF}), 
\begin{equation} \label{deltaht}
\begin{split}
\lefteqn{a^{ij} (\delta h)_{ij} + b^i (\delta h)_{i} + f \delta h} \\ 
= {} & \frac{1}{(v+g_iv_i)} \left(\frac{\partial F}{\partial A_{ij}} \delta g_{ij} + \frac{\partial F}{\partial p_i} \delta g_i + \frac{\partial F}{\partial u} \delta g \right)  - \frac{F v_i \delta g_i}{(v+g_i v_i)^2}.
\end{split}
\end{equation}
We can now establish (i).  From (\ref{gi0}),  (\ref{and}) and (\ref{deltaht}), the only contribution to $a^{nn}$ comes from $\delta \Psi^k_{ij}$, where the variation lands on $h_{\ell q}$, giving at $x_0$,
\begin{equation} \label{formulaann}
a^{nn} = \frac{1}{g_n} \frac{\partial F}{\partial A_{ij}} \Psi^n_j \Psi^n_i \Psi^n_n v_n (1+h),
\end{equation}
as required.

From now on, we work at the point $x_0$ at time $t=0$.  Then we have $h=0$ and $h_i=0$ for $i=1, \ldots, n-1$.  Moreover, $\Phi^i_{\ell}$ and $\Psi^k_i$ are diagonal and $\Psi^a_b = \Phi^a_b=\delta_{ab}$ for  $1\le a, b \le n-1$.  We also have $g_n = \Psi^n_n v_n$ and (\ref{formulaann}) becomes
\begin{equation} \label{formulaann2}
a^{nn} = \frac{\partial F}{\partial A_{nn}}  (\Psi^n_n)^2.
\end{equation}
 
Compute
\begin{equation} \label{Psikij}
\begin{split}
\delta \Psi^k_{ij} \approx {} &  \Psi^q_j \Psi^{\ell}_i \Psi^k_p ((\delta h)_{\ell q} v_p + (\delta h)_{\ell} v_{pq} + (\delta h)_q v_{p\ell})  \\
{} &+ (\Psi^r_j \Psi^q_s  \Psi^{\ell}_i \Psi^k_p + \Psi^q_j \Psi^r_i \Psi^{\ell}_s \Psi^k_p + \Psi^q_j \Psi^{\ell}_i \Psi^r_p \Psi^k_s  ) (\delta h)_r v_s (h_{\ell q} v_p + h_{\ell} v_{pq} + h_q v_{p\ell}) \\
\approx {} &  \Psi^q_j \Psi^{\ell}_i \Psi^k_n (\delta h)_{\ell q} v_n  +  \Psi^q_j \Psi^{n}_i \Psi^k_p (\delta h)_n v_{pq} + \Psi^n_j \Psi^{\ell}_i \Psi^k_p(\delta h)_n v_{p\ell}  \\
{} &+ \bigg(  \Psi^n_j \Psi^q_n \Psi^{\ell}_i \Psi^k_p +  \Psi^q_j \Psi^n_i \Psi^{\ell}_n \Psi^k_p + \Psi^q_j \Psi^{\ell}_i \Psi^n_p \Psi^k_n    \bigg) (\delta h)_n  v_n (h_{\ell q} v_p + h_{\ell} v_{pq} + h_q v_{p\ell}).
\end{split}
\end{equation}
Observe in particular that if $1 \le i,j \le n-1$ then
$$\delta \Psi^k_{ij} \approx \delta^k_n (\Psi_n^n)^2 (\delta h)_n (v_n)^2 h_{ij},$$
so that from (\ref{gi0}),
\begin{equation} \label{deltagij0}
\begin{split}
\delta g_{ij} \approx {} &  (\Psi_n^n)^2 (\delta h)_n (v_n)^3 h_{ij}, \quad \textrm{for } 1 \le i, j \le n-1, \\
\delta g_i \approx {} & \Psi^j_i \Psi^k_{\ell} (\delta h)_j v_{\ell} v_k, \quad \textrm{for } 1 \le i \le n.
\end{split} 
\end{equation}
Hence,
\begin{equation*}
\begin{split}
\frac{\partial F}{\partial A_{ij}} \delta g_{ij} \approx {} & \sum_{i,j=1}^{n-1} \frac{\partial F}{\partial A_{ij}} \delta g_{ij} +  2\sum_{i=1}^{n-1} \frac{\partial F}{\partial A_{in}} \delta g_{in} +  \frac{\partial F}{\partial A_{nn}} \delta g_{nn} \\
\approx {} & \sum_{i,j=1}^{n-1}  \frac{\partial F}{\partial A_{ij}} (\Psi^n_n)^2 (v_n)^3 h_{ij} (\delta h)_n 
+ 2 \sum_{i=1}^{n-1} \frac{\partial F}{\partial A_{in}}  (\Psi^n_n)^2 v_n^2 (\delta h)_{in}  \\ {} & + \frac{\partial F}{\partial A_{nn}}  (\Psi^n_n)^3 v_n^2 (\delta h)_{nn},
\end{split}
\end{equation*}
using  (\ref{gi0}), (\ref{Psiki}), (\ref{Psikij}) and (\ref{deltagij0}).

We now put this all together, noting that at $x_0$ and $t=0$ we have $g_n=\Psi^n_nv_n$ and so $g_iv_i = \Psi^n_n v_n^2$.  From (\ref{deltaht}),
\begin{equation*}
\begin{split}
\lefteqn{a^{ij} (\delta h)_{ij} + b^i (\delta h)_{i} + f \delta h} \\ \approx {} &  \frac{1}{\Psi^n_n v_n^2} \bigg( \sum_{i,j=1}^{n-1} \frac{\partial F}{\partial A_{ij}} (\Psi^n_n)^2 (v_n)^3 h_{ij} (\delta h)_n + 2 \sum_{i=1}^{n-1} \frac{\partial F}{\partial A_{in}} (\Psi^n_n)^2 v_n^2 (\delta h)_{in} \\ {} &  + \frac{\partial F}{\partial A_{nn}} (\Psi_n^n)^3 v_n^2 (\delta h)_{nn}  + \frac{\partial F}{\partial p_n} (\Psi^n_n)^2 v_n^2 (\delta h)_n \bigg)   - \frac{F v_n (\Psi^n_n)^2 v_n^2 (\delta h)_n}{(\Psi_n^n v_n^2)^2} \\
\approx {} & \sum_{i,j=1}^{n-1}  \frac{\partial F}{\partial A_{ij}}  \Psi^n_n v_n h_{ij} (\delta h)_n + 2\sum_{i=1}^{n-1}  \frac{\partial F}{\partial A_{in}} \Psi^n_n (\delta h)_{in} + \frac{\partial F}{\partial A_{nn}} (\Psi_n^n)^2 (\delta h)_{nn}  \\ {} & + \frac{\partial F}{\partial p_n} \Psi^n_n (\delta h)_n - \frac{F}{v_n}  (\delta h)_n.
 \end{split}
\end{equation*}
The theorem follows immediately from this.
	\end{proof}

		\section{The $p$-Laplacian evolution equation} \label{sectionp}
	
	In this section, we give the proof of Theorem \ref{thmplap}.  We wish to construct a smooth function $g(x,t)$ on a neighborhood of the boundary $\partial \Omega_t$ of the moving sets $\ov{\Omega}_t$, solving  the nonlinear equation  
		$$\frac{\partial g}{\partial t}=F(\nabla^2 g, \nabla g, g) := \left( \frac{p-2}{p-1} \right) g \left( | \nabla g|^{p-2} \Delta g + (p-2) |\nabla g|^{p-4}g_ig_j g_{ij} \right)  + |\nabla g|^p.$$
As in Section \ref{Sectionlinearization}, we transform this equation to one on a neighborhood of the boundary $\partial \Omega_0$ in the fixed set $\ov{\Omega}_0$.

	  In order to apply Theorem \ref{CWthm}, we need to check the conditions on the coefficients $a^{ij}$, $b$ at points of $\partial \Omega_0$.  We work at a fixed point $x_0 \in \partial \Omega_0$.  Recall that we are linearizing (\ref{dhdt1}) at a function $h=\overline{h}+\zeta$ for $\zeta \in \mathcal{D}$ as described in \S2 while $g$ is the corresponding function defined by \eqref{h0}.
	  From part (i) of Theorem \ref{thmtechnical}, we immediately obtain $a^{ij} \nu_i \nu_j=0$ from (\ref{deg}), giving (A).
	  
	  To establish  (B) we
	  work now at $t=0$, using the same coordinates centered at $x_0$ and notation of the previous section.   We note that since $v=g_0$ here, we have at this point $\Psi_i^j=\delta_{ij} = \Phi^i_j$ and $g_n=v_n<0$.  First note that 	\begin{equation} \label{le1}
	\begin{split}
\partial_n \frac{\partial F}{\partial A_{nn}}  
	= {} & \left( \frac{p-2}{p-1} \right) g_n   ( |g_n|^{p-2} + (p-2) |g_n|^{p-2} ) \\
	= {} & (p-2) g_n |g_n|^{p-2},
	\end{split}
	\end{equation}
using Theorem \ref{thmtechnical}.

	By part (iv) of Theorem \ref{thmtechnical} and (\ref{le1}),
	\begin{equation} \label{le3}
	\begin{split}
	(b^i - \sum_{i=1}^n \partial_i a^{ji}) \nu_j 	= {} & \frac{\partial F}{\partial p_n}  - \frac{F}{v_n} - 2\sum_{i=1}^{n-1} \left( \partial_i \frac{\partial F}{\partial A_{in}} \right)  - \left( \partial_n \frac{\partial F}{\partial A_{nn}} \right)  \\
	= {} & p |g_n|^{p-2}g_n  - \frac{|g_n|^p}{v_n}  - (p-2) |g_n|^{p-2} g_n  \\
	= {} & |g_n|^{p-2}g_n  <0,
	\end{split}
	\end{equation}
	where we used the fact that $\partial_i \frac{\partial F}{\partial A_{in}}=0$ for $i=1, \ldots, n-1$ and (\ref{le1}).

Since we have strict inequality in (\ref{le3}), by continuity this holds for $t$ in $[0,T]$, if $T>0$ is sufficiently small, giving  (B).

We can now complete the proof of Theorem \ref{thmplap}.  We would like to extend our operator $G$ into the interior of $\Omega_0$.  As noted in the introduction, the operator $F$ is degenerate not just along the moving boundary, but also at points in the interior where $\nabla g=0$  vanishes.  To get around this problem, we make a change to the operator in the interior as follows.

Under the correspondence (\ref{h0}) of Section \ref{Sectionlinearization}, $g$ maps to a continuous function $h$ on $(\ov{\Omega} \setminus \Omega_0^{(\eta)}) \times [0,T]$ for a small $\eta>0$.  Moreover, $h$ is smooth on $(\Omega_0 \setminus \Omega_0^{(\eta)}) \times [0,T]$ and satisfies
 $$\frac{\partial h}{\partial t} = G(D^2 h, Dh, h, x),$$
 there, with $h=0$ at $t=0$.
We now extend $h$ over the set $\ov{\Omega}_0^{(\eta)} \times[0,T]$ to obtain a function $\tilde{h}$ with the following properties:  $\tilde{h}$ is continuous on $\ov{\Omega}_0  \times [0,T]$ and smooth on $\Omega_0  \times [0,T]$; $\tilde{h}=h$ on $(\ov{\Omega}_0 \setminus \Omega_0^{(\eta)}) \times [0,T]$; and $\tilde{h}(x,0)=0$ for $x \in \ov{\Omega}_0$.

We also extend the operator $G$ to be a smooth operator $\tilde{G}$ on for $x \in \ov{\Omega}_0$ which is uniformly elliptic on $\ov{\Omega}_0^{(\eta)}$ and has $G=\tilde{G}$ on $\ov{\Omega}_0 \setminus \Omega_0^{(\eta)}$.  
  Let $\psi$ be a smooth cut-off function which is equal to $1$ on $\ov{\Omega}_0^{(\eta)}$ and equal to $0$ on $\ov{\Omega}_0 \setminus \Omega_0^{(\eta/2)}$.  Define an operator
 $$\hat{G}(D^2 h, Dh, h, x):=
\tilde{G}(D^2 h, Dh, h,x) + \psi \tilde{f},$$
where $\tilde{f}$ is the function $$\tilde{f}:=\frac{\partial \tilde{h}}{\partial t} - \tilde{G}(D^2 \tilde{h}, D\tilde{h}, \tilde{h},x).$$

Then $\hat{G}$ is a smooth operator, uniformly elliptic in $\Omega_0^{(\eta/2)}$ and coincides with $G$ in $\ov{\Omega}_0 \setminus \Omega_0^{(\eta/2)}$.  From what we have proved above, $\hat{G}$ satisfies all the conditions of Theorem \ref{CWthm}.  Applying this theorem we obtain $\ve>0$ and a smooth solution $\hat{h} \in C^{\infty}(\ov{\Omega}_0 \times [0,\ve])$ of the equation
\begin{equation} \label{hat}
\begin{split}
\frac{\partial \hat{h}}{\partial t} = {} & \hat{G}(D^2 \hat{h}, D\hat{h}, \hat{h}, x) \\
\hat{h}(x,0) = {} & 0, \quad x \in \ov{\Omega}_0.
\end{split}
\end{equation}
On the other hand we observe that $\tilde{h}$, which  is smooth on $\Omega_0 \times [0,\ve]$ (shrinking $\ve>0$ if necessary), solves (\ref{hat}) there by the construction of $\hat{g}$ and $\tilde{h}$.  We claim that $\hat{h}=\tilde{h}$.  Indeed, via (\ref{h0}) of Section \ref{Sectionlinearization}, $\hat{h}$ and $\tilde{h}$ correspond to functions $\hat{g}$ and $\tilde{g}$ which satisfy an equation which coincides with (\ref{plapg}) near boundaries of $\Omega_t$ in $\mathbb{R}^n$ and satisfy a nondegenerate parabolic equation away from the boundary.  The same argument for uniqueness of continuous solutions of the $p$-Laplacian evolution equation \cite{DBH} gives uniqueness for this modified equation.  This proves the claim that $\hat{h} =\tilde{h}$.  But $\tilde{h}=h$ is smooth near the boundary of $\Omega_0$.  Since $h$ corresponds to $g$ near the boundary, the boundaries $\partial \Omega_t$ are smooth and $g(x,t)$ is smooth function in a neighborhood of the boundary $\cup_{t\in [0,T]} \partial \Omega_t \times \{ t \}$, as required.

	\section{The  $\alpha$-Gauss curvature flow with a flat side} \label{SectionGCF}
	
	In this section we prove Theorem \ref{thmGCF}, by using the same strategy as in the last section now to the case of the $\alpha$-Gauss curvature flow.  
	
	We first prove (A) that $a^{ij}\nu_i \nu_j=0$ at every  $x_0 \in \partial \Omega_0$ for $t\in [0,T]$.   Recall that to apply Theorem  \ref{CWthm} we linearize the operator $G$ at a function $h= \ov{h} + \zeta$ for $\zeta \in \mathcal{D}$.  The function $h$ defines a family of diffeomorphisms $\Phi_t(x)=\Phi(x,t)$ and functions $g(y,t)$ on domains $\Phi_t (\ov{\Omega}_0 \setminus \Omega_0^{(\eta)}) \subset \mathbb{R}^n$.
	 By (i) of Theorem \ref{thmtechnical} and (\ref{gi0}), it suffices to show that at this $g$ we have
	\begin{equation} \label{nut}
	\nu_t^i \frac{\partial F}{\partial A_{ij}}=0,
	\end{equation}
	where $\nu_t$ is the outward pointing unit normal to $\Phi_t (\ov{\Omega}_0 \setminus \Omega_0^{(\eta)})$ at $\Phi(x_0,t)$. 
	
	 We can choose coordinates centered at $\Phi(x_0,t)$ so that the  boundary $\Phi_t (\partial \Omega_0)$ is given locally by $y_n = \rho(y')$, for $y'=(y_1, \ldots, y_{n-1})$, with $\rho(0)=0$ and $\nabla \rho(0) =0$.  We may assume after a rotation that $(\rho_{ij})_{i,j=1}^{n-1}$ is diagonal at this point.  Since $g(y', \rho(y'))=0$ it follows that $g_i=0$ for $i=1, \ldots, n-1$ and $g_{ij}=0$ for $i,j=1, \ldots, n-1$ with $i\neq j$ at this point.

Write $B=(B_{ij})$ for the $n \times n$ matrix with entries
$$B_{ij} =  g^{\frac{1}{n}} g_{ij} + \frac{1}{\sigma} g^{\frac{1}{n}-1} g_i g_j,$$
so that from (\ref{eqnGCF2}),
$$F(\nabla^2g, \nabla g, g)= \frac{(\det B)^{\alpha}}{(1+ g^{\frac{2}{\sigma}} |\nabla g|^2)^{\frac{(n+2)\alpha -1}{2}}}.$$
To compute the determinant of $B$, it is helpful to recall that, writing $B' = (B_{ij})_{i,j=1}^{n-1}$, and $Y = (B_{n1}, \ldots, B_{n,n-1})$, we have $$\det B = (B_{nn} - Y (B')^{-1} Y^T) \det B'.$$ A short calculation gives \[
\begin{split}
\det B = {} & g g_{11}\cdots g_{nn} + \frac{1}{\sigma} g_{11} \cdots g_{n-1, n-1} g_n^2 \\
{} & -  \frac{1}{\sigma}\sum_{i=1}^{n-1} g_{11} \cdots \widehat{g_{ii}} \cdots g_{n-1,n-1}  (g_{in}+ g_{ni}) g_i g_n  + E(\nabla^2g, \nabla g, g),
\end{split}
\]
where the ``error term'' $E$ has the property that at $\Phi(x_0, t)$,
$$0 = E= \partial_n \frac{\partial E}{\partial A_{nn}} = \frac{\partial E}{\partial p_n} = \frac{\partial E}{\partial A_{ii}} = \partial_i \frac{\partial E}{\partial A_{in}},$$
for $i=1, \ldots, n-1$.  In particular, terms involving $B_{ij}$ for $i\neq j$ and $i,j=1, \ldots, n-1$ are all absorbed in $E$. It follows immediately that $\frac{\partial F}{\partial A_{nn}}=0$, which gives (\ref{nut}) by the Cauchy-Schwarz inequality.

To establish (B) we work at time $t=0$. At $x_0$ and $t=0$ we choose coordinates $x_1, \ldots x_n$ centered at $x_0$ so that the boundary of $\Omega_0$ is given locally by $x_n = \rho(x')$ with $\rho(0)=0$ and $\nabla \rho(0)=0$.   Since $\Omega_0$ is strongly convex for $t$ small, we have $(D^2\rho)(0) <0$, and we also assume that $(\rho_{ij})_{i,j=1}^{n-1}$ is diagonal at this point.   Since $h$ and $v$ vanish on $\partial \Omega_0$ we have $h=0=v$ and $h_i=0=v_i$ for $i=1, \ldots, n-1$, and $v_n, h_n<0$.     Moreover, 
\begin{equation} \label{hij3}
h_{ij} = - \rho_{ij} h_n,  \quad v_{ij} = - \rho_{ij}v_n, \quad  i,j=1, \ldots, n-1.
\end{equation}
In particular, the matrices $(h_{ij})_{i,j=1}^{n-1}$ and $(v_{ij})_{i,j=1}^{n-1}$ are diagonal and negative definite.  On the other hand, using these coordinates in the complement of $\Omega_0$ at $t=0$, by construction $g_n>0$ at $x_0$, and since $g$ vanishes along $\partial \Omega_0$,
\begin{equation} \label{rhoij}
g_{ij} = - \rho_{ij} g_n.
\end{equation}
In particular, $(g_{ij})_{i,j=1}^{n-1}$ is diagonal and positive definite.

In what follows, we apply Theorem \ref{thmtechnical} repeatedly.  We have at $x_0$ and $t=0$, 
\[
\begin{split}
\frac{\partial F}{\partial p_n} = {} & \alpha (\det B)^{\alpha -1} g_{11}\cdots g_{n-1, n-1} \frac{2}{\sigma} g_n \\
= {} &  \frac{2\alpha}{\sigma^{\alpha}} (g_{11}\cdots g_{n-1, n-1})^{\alpha} (g_n^2)^{\alpha-1} g_n. \\
\end{split}
\]
Next, compute for $i=1, \ldots, n-1$,
$$\frac{\partial F}{\partial A_{ii}} = \alpha (\det B)^{\alpha-1} \frac{1}{\sigma} g_{11} \cdots \widehat{g_{ii}} \cdots g_{n-1,n-1} g_n^2.$$
From (\ref{gi0}), we have
$$v_n = \Phi^n_n g_n = (1-h_nv_n) g_n,$$
and hence
$$h_n = \frac{1}{v_n} ( 1- \Phi^n_n).$$
Then from (\ref{hij3}) and (\ref{rhoij}),
$$h_{ii} = g_{ii} (1-\Phi^n_n) \frac{\Phi^n_n}{v_n^2}.$$
Hence, recalling that $v_n<0$ and $\Phi^n_n<0$,
\begin{equation} \label{binui}
\begin{split}
b^n = {} & \sum_{i=1}^{n-1} \frac{\partial F}{\partial A_{ii}}  \Psi^n_n v_n h_{ii}  +\frac{\partial F}{\partial p_n} \Psi^n_n - \frac{F}{v_n} \\
= {} & \sum_{i=1}^{n-1} \frac{\alpha}{\sigma^{\alpha-1}} (g_{11}\cdots g_{n-1,n-1})^{\alpha-1} (g_n^2)^{\alpha-1} \frac{1}{\sigma} g_{11} \cdots g_{n-1,n-1} g_n^2 (1-\Phi^n_n) \frac{\Phi^n_n}{v_n^2} \Psi^n_n v_n \\
{} & +  \frac{2\alpha}{\sigma^{\alpha}} (g_{11}\cdots g_{n-1, n-1})^{\alpha} (g_n^2)^{\alpha-1} (\Psi^n_n)^2 v_n - \frac{\left( \frac{1}{\sigma} g_{11} \cdots g_{n-1,n-1} g_n^2 \right)^{\alpha}}{v_n}  \\
= {} & \frac{(n-1)\alpha}{\sigma^{\alpha}} (g_{11} \cdots g_{n-1,n-1})^{\alpha} (g_n^2)^{\alpha-1} (\Psi^n_n)^2 v_n (1-\Phi^n_n) 
\\ {} &  + \frac{1}{\sigma^{\alpha}} (2 \alpha -1) \left( g_{11} \cdots g_{n-1,n-1}  \right)^{\alpha}  (g_n^2)^{\alpha-1} (\Psi^n_n)^2 v_n \\
< {} & \frac{1}{\sigma^{\alpha}} ((n+1) \alpha -1) \left( g_{11} \cdots g_{n-1,n-1}  \right)^{\alpha}  (g_n^2)^{\alpha-1} (\Psi^n_n)^2 v_n,
\end{split} 
\end{equation}
since $\alpha>\frac{1}{n+1}$, where we discarded the term with $\Phi^n_n$.
	
Next, we have \begin{equation} \label{alphauno}
\begin{split}
 \partial_n \frac{\partial F}{\partial A_{nn}} = {} & \alpha (\det B)^{\alpha-1} g_{11} \cdots g_{n-1, n-1} g_n \Phi_n^n \\
= {} & \frac{\alpha}{\sigma^{\alpha-1}} (g_{11} \cdots g_{n-1,n-1})^{\alpha} (g_n^2)^{\alpha-1} g_n \Phi^n_n.
\end{split}
\end{equation}
and
for each $i=1, \ldots, n-1$,
\[
\begin{split}
\partial_i \frac{\partial F}{\partial A_{in}} = {} & - \alpha (\det B)^{\alpha-1} g_{11}\cdots g_{n-1, n-1} \frac{1}{\sigma} g_n \\
= {} & -\frac{\alpha}{\sigma^{\alpha}} (g_{11} \cdots g_{n-1,n-1})^{\alpha} (g_n^2)^{\alpha-1} \Psi^n_n v_n.
\end{split}
\]
Then, using (\ref{binui}) and (\ref{alphauno}),
\begin{equation} \label{alphatres}
\begin{split}
b^n - \sum_{i=1}^n \partial_i a^{ni} < {} &  \frac{1}{\sigma^{\alpha}} ((n+1) \alpha -1) \left( g_{11} \cdots g_{n-1,n-1}  \right)^{\alpha}  (g_n^2)^{\alpha-1} (\Psi^n_n)^2 v_n \\ {} & - 2\sum_{i=1}^{n-1} \left( \partial_i \frac{\partial F}{\partial A_{in}} \right) \Psi^n_n  - \left( \partial_n \frac{\partial F}{\partial A_{nn}} \right) (\Psi^n_n)^2 \\
= {} & \frac{1}{\sigma^{\alpha}} ((n+1) \alpha -1) \left( g_{11} \cdots g_{n-1,n-1}  \right)^{\alpha}  (g_n^2)^{\alpha-1} (\Psi^n_n)^2 v_n \\ {} &    +  \frac{2 (n-1) \alpha}{\sigma^{\alpha}} (g_{11} \cdots g_{n-1,n-1})^{\alpha} (g_n^2)^{\alpha-1} (\Psi^n_n)^2 v_n \\ {} &  - \frac{\alpha}{\sigma^{\alpha-1}} (g_{11} \cdots g_{n-1,n-1})^{\alpha} (g_n^2)^{\alpha-1}  (\Psi^n_n)^2 v_n \\
= {} & \frac{\alpha(2n-1)}{\sigma^{\alpha}} (g_{11} \cdots g_{n-1,n-1})^{\alpha} (g_n^2)^{\alpha-1}  (\Psi^n_n)^2 v_n<0,
\end{split}
\end{equation}
where for the last equality, we recall that $\sigma = n- \frac{1}{\alpha}$.
	
	We have strict inequality in (\ref{alphatres}) so by continuity this holds for $t \in [0,T]$ if $T>0$ is sufficiently small, giving  (B).  We now apply the same argument as in Section \ref{sectionp} above.  In this case, the uniqueness $\hat{h}=\tilde{h}$ follows from uniqueness of viscosity solutions of degenerate parabolic equations of the type of the $\alpha$-Gauss curvature flow (see \cite[Section 8]{CIL}).


\begin{thebibliography}{0}
	\bibitem{A}  Andrews, B., \emph{Gauss curvature flow: the fate of the rolling stones}, Invent. Math. 138
(1999), no. 1, 151--161
\bibitem{Andrews} Andrews, B., \emph{Motion of hypersurfaces by Gauss curvature}, Pacific J. Math. 195
(2000), no.1, 1--34
\bibitem{ACGL} Andrews, B.,  Chow, B., Guenther, C., Langford, M., \emph{Extrinsic geometric flows}, 
Grad. Stud. Math., 206
American Mathematical Society, Providence, RI, 2020, xxviii+759 pp
\bibitem{BCD} Brendle, S.,  Choi, K., Daskalopoulos, P., \emph{Asymptotic behavior of flows by powers
of the Gaussian curvature}, Acta Math. 219 (2017), 1--16
	\bibitem{CW} Chau, A., Weinkove, B., \emph{Smooth solutions of degenerate linear parabolic equations and the porous medium equation}, preprint, arXiv:2311.14522
\bibitem{CK} Choe, H. J., Kim, J. \emph{Regularity for the interfaces of evolutionary $p$-Laplacian functions}, SIAM J. Math. Anal. 26 (1995), no.4, 791--819
\bibitem{Chow} Chow, B. {\em Deforming convex hypersurfaces by the $n$th root of the Gaussian curvature}, J. Differ. Geom. 22 (1985), 117–138.
\bibitem{CIL} Crandall, M.G., Ishii, H., Lions, P.-L., \emph{User's guide to viscosity solutions of second order partial differential equations},  Bull. Amer. Math. Soc. (N.S.) 27 (1992), no. 1, 1--67
\bibitem{DH98} Daskalopoulos, P., Hamilton, R. {\em Regularity of the free boundary for the porous medium equation},
J. Amer. Math. Soc. 11 (1998), no. 4, 899--965
\bibitem{DH}  Daskalopoulos, P., Hamilton, R., \emph{The free boundary in the Gauss curvature flow
with flat sides}, J. Reine Angew. Math. 510 (1999), 187–22
\bibitem{DL}  Daskalopoulos, P., Lee, K.-A., \emph{Worn stones with flat sides all time: regularity of the
interface}, Invent. Math. 156 (2004), 445--493
\bibitem{DS} Daskalopoulos, P., Savin, O., \emph{$C^{1, \alpha}$ regularity of solutions to parabolic Monge-Amp\`ere equations},  Amer. J. Math. 134 (4) (2014), 1051--1087
		\bibitem{DB} DiBenedetto, E., \emph{Degenerate parabolic equations}, Springer-Verlag, New York 1993, xvi+387 pp.
		\bibitem{DBH} DiBenedetto, E.,  Herrero, M.A., \emph{On the Cauchy problem and initial traces for a
degenerate parabolic equation}, Trans. Amer. Math. Soc. 314 (1989), no. 1, 187–224.
\bibitem{Fi56}  Fichera, G. {\em Sulle equazioni differenziali lineari ellittico-paraboliche del secondo ordine},
Atti Accad. Naz. Lincei Mem. Cl. Sci. Fis. Mat. Natur. Sez. Ia (8) 5 (1956), 1--30
\bibitem{Hamilton}  Hamilton, R. S., \emph{Worn stones with flat sides}, Discourses Math. Appl., 3
Texas A \& M University, Department of Mathematics, College Station, TX, 1994, 69--78
\bibitem{Hanz} Hanzawa, E. {\em Classical solutions of the Stefan problem}, Tohoku Math. J. (2) 33 (1981), no. 3, 297--335
\bibitem{HWZ} Huang, G., Wang, X.-J., Zhou, Y., \emph{Long time regularity of the $p$-Gauss curvature flow with flat side}, preprint, arXiv:2403.12292
\bibitem{HWZ2} Huang, G., Wang, X.-J., Zhou, Y., \emph{Regularity of the $p$-Gauss curvature flow with flat side}, preprint, arXiv:2407.05663
\bibitem{JL} Jang, H. S., Lee, K.-A., \emph{H\"older regularity of solutions of degenerate parabolic equations of general dimension}, preprint, arXiv: 2412.00675
\bibitem{KLR} Kim, L., Lee, K.-A., Rhee, F., \emph{$\alpha$-Gauss curvature flow with flat sides}, J. Diff. Equations
254 (2013), no. 3, 1172--1192
\bibitem{Ko} Ko, Y., \emph{$C^{1, \alpha}$-regularity of interfaces for solutions of the parabolic $p$-Laplacian equation}, Comm. Partial Differential
Equations 24 (5-6) (1999), 915--950
\bibitem{Koch} Koch, H., {\em Non-Euclidean singular integrals and the porous medium equation}, Habilitation
thesis, University of Heidelberg, 1999
\bibitem{KN} Kohn, J.J., Nirenberg, L. {\em Degenerate elliptic-parabolic equations of second order}, Comm. Pure Applied Math. 20 (1967), 797--872
\bibitem{LPV} Lee, K.-A., Petrosyan, A., V\'azquez, J. L., \emph{Large-time geometric properties of solutions of the evolution $p$-Laplacian equation}, 
J. Differential Equations 229 (2006), no. 2, 389--411
\bibitem{Tso} Tso, K.,  {\em Deforming a hypersurface with prescribed Gauss-Kronecker curvature},
Comm. Pure Appl. Math. 38 (1985), 867--882
\bibitem{ZY} Zhao, J. N., Yuan, H. J., {\em Lipschitz continuity of solutions and interfaces of the evolution  $p$-Laplacian equation},
Northeast. Math. J. 8 (1992), no. 1, 21--37
	\end{thebibliography}
\end{document}